\documentclass[leqno]{amsart}
\usepackage{amsmath}
\usepackage{amssymb}
\usepackage{amsthm}
\usepackage{enumerate}
\usepackage[mathscr]{eucal}
\usepackage{enumitem}
\theoremstyle{plain}
\newtheorem{theorem}{Theorem}[section]
\newtheorem{lemma}[theorem]{Lemma}

\theoremstyle{definition}
\newtheorem{definition}[theorem]{Definition}
\newtheorem{remark}[theorem]{Remark}

\newtheorem{cor}[theorem]{Corollary}
\theoremstyle{remark}
\usepackage{amsmath, amsthm, amscd, amsfonts, amssymb, graphicx, color}
\usepackage[bookmarksnumbered, colorlinks, plainpages]{hyperref}
\hypersetup{colorlinks=true,linkcolor=red, anchorcolor=green, citecolor=cyan, urlcolor=red, filecolor=magenta, pdftoolbar=true}
\begin{document}

\title [On generalized Davis-Wielandt radius inequalities]{On generalized Davis-Wielandt radius inequalities of semi-Hilbertian space operators} 

\author[ Aniket Bhanja, Pintu Bhunia and  Kallol Paul]{Aniket Bhanja, Pintu Bhunia and Kallol Paul}

\address[Bhanja] {Department of Mathematics, Vivekananda College Thakurpukur, Kolkata, West Bengal, India}
\email{aniketbhanja219@gmail.com}

\address[Bhunia] {Department of Mathematics, Jadavpur University, Kolkata 700032, West Bengal, India}
\email{pintubhunia5206@gmail.com}

\address[Paul] {Department of Mathematics, Jadavpur University, Kolkata 700032, West Bengal, India}
\email{kalloldada@gmail.com}

%\thanks will become a 1st page footnote.
\thanks{Pintu Bhunia would like to thank UGC, Govt. of India for the financial support in the form of SRF. Prof. Paul would like to thank RUSA 2.0, Jadavpur University for  partial support. }
\thanks{}
%\thanks{}

    %General info


%\date{}
\maketitle
\begin{abstract}
Let $A$ be a positive (semidefinite) operator on a complex Hilbert space $\mathcal{H}$ and let $\mathbb{A}=\left(\begin{array}{cc}
	A & O\\
	O & A
	\end{array}\right).$
We obtain upper and lower bounds for the $A$-Davis-Wielandt radius of semi-Hilbertian space operators, which generalize and improve on the existing ones.  We also obtain  upper bounds for the $\mathbb{A}$-Davis-Wielandt radius of $2 \times 2$ operator matrices. Finally, we determine the exact value for the $\mathbb{A}$-Davis-Wielandt radius of two operator matrices $\left(\begin{array}{cc}
I & X\\
0 & 0
\end{array}\right)$ and $\left(\begin{array}{cc}
0 & X\\
0 & 0
\end{array}\right)$, where $X $ is a semi-Hilbertian space operator.

\end{abstract}

\section{\textbf{Introduction and Preliminaries}}
\noindent 
Let $\mathcal{B}(\mathcal{H})$ denote the $C^{\ast}$-algebra of all bounded linear operators acting on a complex Hilbert space $\mathcal{H}$ with usual  inner product $\langle \cdot, \cdot \rangle$ and the corresponding norm $\|\cdot\| $. The letters $I$ and $O$ stand for the identity operator and the zero operator on $\mathcal{H}$, respectively. For $T\in \mathcal{B}(\mathcal{H})$, we denote by $\mathcal{R}(T)$ and $\mathcal{N}(T)$ the range and the null space of $T$, respectively. By $\overline{\mathcal{R}(T)}$ we denote the norm closure of $\mathcal{R}(T)$. Let $T^*$ be the adjoint of $T$. 
%and $|T|$ be the absolute value of $T$, i.e., $|T|=(T^*T)^{\frac{1}{2}}.$  
The cone of all positive (semidefinite) operators is given by:
$$\mathcal{B}(\mathcal{H})^+=\left\{A\in \mathcal{B}(\mathcal{H})\,:\,\langle Ax, x\rangle\geq 0,\; \forall\;x\in \mathcal{H}\;\right\}.$$
Every $A\in \mathcal{B}(\mathcal{H})^+$ defines the following positive semidefinite sesquilinear form:
$$\langle\cdot,\cdot\rangle_{A}:\mathcal{H}\times \mathcal{H}\longrightarrow\mathbb{C},\;(x,y)\longmapsto\langle x, y\rangle_{A} =\langle Ax, y\rangle,$$
and the seminorm induced by the above sesquilinear form is given by:
 $$\|x\|_A=\sqrt{\langle x, x\rangle_A},~~ x\in \mathcal{H}.$$ 
This makes $\mathcal{H}$ into a semi-Hilbertian space. It is easy to observe that $\|x\|_A=0$ if and only if $x\in \mathcal{N}(A)$. Therefore, $\|\cdot\|_A$ is a norm on $\mathcal{H}$ if and only if $A$ is injective. Also we observe that $(\mathcal{H},\|\cdot\|_A)$ is complete if and only if $\mathcal{R}(A)$ is closed in $\mathcal{H}$. Let us fix the alphabet $A$ for positive (semidefinite) operator on $\mathcal{H}$ and  we also fix $\mathbb{A}=\left(\begin{array}{cc}
	A & O\\
	O & A
	\end{array}\right).$ 
\begin{definition}
Let $T \in \mathcal{B}(\mathcal{H})$. An operator $S\in\mathcal{B}(\mathcal{H})$ is called an $A$-adjoint of $T$ if the equality $\langle Tx, y\rangle_A=\langle x, Sy\rangle_A$ holds, for all $x,y\in \mathcal{H}$.    
\end{definition}
\noindent Therefore, $S$ is an $A$-adjoint of $T$ if and only if $S$ is a solution of the equation $AX=T^*A$ in $\mathcal{B}(\mathcal{H})$. For $T \in \mathcal{B}(\mathcal{H})$, the existence of an $A$-adjoint of $T$ is not guaranteed. The set of all operators acting on $\mathcal{H}$ that admit $A$-adjoints is denoted by $\mathcal{B}_{A}(\mathcal{H})$. It follows from Douglas Theorem \cite{doug} that
\begin{eqnarray*}
\mathcal{B}_{A}(\mathcal{H})& =& \left\{T\in \mathcal{B}(\mathcal{H})\,:\;\mathcal{R}(T^{*}A)\subseteq \mathcal{R}(A)\right\}.
 \end{eqnarray*}
\noindent By Douglas Theorem \cite{doug}, we have  if $T\in \mathcal{B}_A(\mathcal{H})$ then the operator equation $AX=T^*A$ has a unique solution, denoted by $T^{\sharp_A}$, satisfying $\mathcal{R}(T^{\sharp_A})\subseteq \overline{\mathcal{R}(A)}$. Note that $T^{\sharp_A}=A^\dag T^*A$, where $A^\dag$ is the Moore-Penrose inverse of $A$ (see \cite{acg2}). Also, we have   $AT^{\sharp_A}=T^*A.$  An operator $T\in \mathcal{B}(\mathcal{H})$ is said to be $A$-bounded if there exists $c>0$ such that $ \|Tx\|_{A} \leq c \|x\|_{A},$ for all $x\in \mathcal{H}. $  We observe that $ \mathcal{B}_{A^{1/2}}(\mathcal{H}) $ is the collection of all $A$-bounded operators, i.e.,
\begin{equation*}
	\mathcal{B}_{A^{1/2}}(\mathcal{H})=\left\{T \in \mathcal{B}(\mathcal{H})\,:\;{\exists} \;c > 0\,\mbox{such that}\;\|Tx\|_{A} \leq c \|x\|_{A},\;\forall\,x\in \mathcal{H}  \right\}.
\end{equation*}

\noindent It is well-known that $\mathcal{B}_{A}(\mathcal{H})$ and $\mathcal{B}_{A^{1/2}}(\mathcal{H})$ are two subalgebras of $\mathcal{B}(\mathcal{H})$ which are neither closed nor dense in $\mathcal{B}(\mathcal{H})$. Moreover, the following inclusions 
$$\mathcal{B}_{A}(\mathcal{H})\subseteq\mathcal{B}_{A^{1/2}}(\mathcal{H})\subseteq \mathcal{B}(\mathcal{H})$$ 
hold with equality if $A$ is injective and has closed range. Let us now define $A$-selfadjoint, $A$-normal and $A$-unitary operators.
  
\begin{definition}
An operator $T\in \mathcal{B}(\mathcal{H})$ is called $A$-selfadjoint if $AT$ is selfadjoint, i.e., $AT=T^*A$ and it is called $A$-positive if $AT\geq 0$. 
\end{definition}
\noindent Observe that if $T$ is $A$-selfadjoint then $T\in \mathcal{B}_{A}(\mathcal{H})$. However, in general, it does not always imply $T=T^{\sharp_A}$. 
An operator $T\in \mathcal{B}_{A}(\mathcal{H})$ satisfies $T=T^{\sharp_A}$ if and only if T is $A$-selfadjoint and $\mathcal{R}(T) \subseteq \overline{\mathcal{R}(A)}$.

\begin{definition}
An operator $T\in \mathcal{B}_A(\mathcal{H})$ is said to be $A$-normal if $TT^{\sharp_A}=T^{\sharp_A}T$. 
\end{definition}
\noindent We know that every selfadjoint operator is normal. But, an $A$-selfadjoint operator is not necessarily $A$-normal (see \cite[Example 5.1]{BFS}).

\begin{definition}
An operator $U\in  \mathcal{B}_A(\mathcal{H})$ is said to be $A$-unitary if $\|Ux\|_A=\|U^{\sharp_A}x\|_A=\|x\|_A$, for all $x\in \mathcal{H}$.
\end{definition}
\noindent It was shown in \cite{acg1} that an operator $U\in  \mathcal{B}_A(\mathcal{H})$ is $A$-unitary if and only if $U^{\sharp_A} U=(U^{\sharp_A})^{\sharp_A} U^{\sharp_A}=P_A,$ where $P_A$ denotes the projection  onto $\overline{\mathcal{R}(A)}$. We mention here that if $T\in \mathcal{B}_A(\mathcal{H})$ then $T^{\sharp_A}\in \mathcal{B}_A(\mathcal{H})$ and $(T^{\sharp_A})^{\sharp_A}=P_ATP_A$.

\noindent Let $T \in \mathcal{B}_{A^{1/2}}(\mathcal{H})$. The $A$-operator seminorm and the $A$-minimum modulus of $T$ are defined respectively as:
\begin{eqnarray*}
	\|T\|_A &=&\sup \left \{\frac{\|Tx\|_A}{\|x\|_A}: ~~{\substack{x\in \overline{\mathcal{R}(A)},~~ x\not=0}}\right \}=\sup\left\{\|Tx\|_{A}\,:\;x\in \mathcal{H},\,\|x\|_{A}= 1\right\},\\
	m_A(T) &=&\inf \left \{\frac{\|Tx\|_A}{\|x\|_A}: ~~{\substack{x\in \overline{\mathcal{R}(A)},~~ x\not=0}}\right \}=\inf\left\{\|Tx\|_{A}\,:\;x\in \mathcal{H},\,\|x\|_{A}= 1\right\}.
\end{eqnarray*}
\noindent Let $T \in \mathcal{B}_{A^{1/2}}(\mathcal{H})$. The $A$-numerical range, the $A$-numerical radius and the $A$-Crawford number of $T$ are defined respectively as:  
\begin{eqnarray*}
W_A(T) &=& \{\langle Tx, x\rangle_A: x\in \mathcal{H}, \|x\|_A=1\},\\
w_A(T)&=&\sup\{|c|: c \in W_A(T)\} ~\mbox{and}\\
c_A(T)&=&\inf\{|c|: c \in W_A(T)\}.
\end{eqnarray*}
The $A$-operator seminorm attainment set of $T,$ denoted as $M^{A}_T$, is defined as the set of all $A$-unit vectors in $\mathcal{H}$ at which $T$ attains its $A$-operator seminorm, i.e., $$ M^{A}_T =\left \{ x\in \mathcal{H} :   \|Tx\|_A=\|T\|_A, \|x\|_A=1   \right \}.$$ Likewise 
the $A$-numerical radius attainment set and the $A$-Crawford number attainment set of $T,$  denoted as $W^{A}_T$ and $c^{A}_T$ respectively, are defined as
\begin{eqnarray*}
	W^{A}_T &=& \left \{ x\in \mathcal{H} :   |\langle Tx,x\rangle_A|=w_A(T), \|x\|_A=1   \right \},\\
	c^{A}_T &=& \left \{ x\in \mathcal{H} :   |\langle Tx,x\rangle_A|=c_A(T), \|x\|_A=1   \right \}. 
\end{eqnarray*}
It is well known that $\|\cdot\|_A$ and $w_A(\cdot)$ are equivalent seminorm on $\mathcal{B}_{A^{1/2}}(\mathcal{H})$, satisfying the following inequality:
\begin{equation*}
\tfrac{1}{2}\|T\|_A \leq w_A(T)\leq \|T\|_A,~~T \in \mathcal{B}_{A^{1/2}}(\mathcal{H}).
\end{equation*}
The first inequality becomes equality if $AT^2=O$  and the second inequality becomes equality if $T$ is $A$-normal (see \cite {F}).
For $T\in B_A(\mathcal{H})$, we write ${Re}_A(T)=\frac{1}{2}(T+T^{\sharp_A})$ and ${Im}_A(T)=\frac{1}{2{\rm i}}(T-T^{\sharp_A})$. For every $A$-selfadjoint operator $T$, we have (see \cite{Z}) $$w_A(T)=\|T\|_A.$$
\noindent Also $T^{\sharp_A}T$, $TT^{\sharp_A}$ are A-selfadjoint and A-positive operators satisfying the following equality:
$$\|T^{\sharp_A}T\|_A =\|TT^{\sharp_A}\|_A =\|T\|^2_A =\|T^{\sharp_A}\|^2_A.$$ 
For $T\in B_A(\mathcal{H})$, we write $|T|^2_A=T^{\sharp_A}T$. For $T,S \in  B_A(\mathcal{H})$,  $(TS)^{\sharp_A}=S^{\sharp_A}T^{\sharp_A}$, $\|TS\|_A\leq \|T\|_A\|S\|_A$ and  $\|Tx\|_A\leq \|T\|_A\|x\|_A$, for all $x\in \mathcal{H}$.  For further readings we refer the readers to \cite{acg1,acg2}.\\

\noindent Motivated by the study of the $A$-numerical radius of semi-Hilbertian space operators, we here study the $A$-Davis-Wielandt radius of semi-Hilbertian space operators. This is a generalization of the Davis-Wielandt radius of Hilbert space operators. The Davis-Wielandt shell and the Davis-Wielandt radius of an operator $T\in \mathcal{B}(\mathcal{H})$ are defined respectively as (see \cite{D,W}): 
\begin{eqnarray*}
 DW(T)& = &\left\{ \left(\langle Tx, x \rangle, \|Tx\|^2 \right) : x\in \mathcal{H}, \|x\|=1 \right\},\\
 dw(T)&= &\sup \left\{ \sqrt{|\langle Tx, x \rangle|^2+ \|Tx\|^4} : x\in \mathcal{H}, \|x\|=1 \right \}.
\end{eqnarray*}
Recently many mathematicians \cite{LP,LPS,LSZ,ZS,ZMCN} have studied the  Davis-Wielandt shell and the Davis-Wielandt radius of an operator $T\in \mathcal{B}(\mathcal{H})$.
The $A$-Davis-Wielandt shell and the $A$-Davis-Wielandt radius of an operator $T\in \mathcal{B}_{A^{1/2}}(\mathcal{H})$ are defined respectively as (see \cite{FA}): 
\begin{eqnarray*}
 DW_A(T)& = &\left\{ \left(\langle Tx, x \rangle_A, \|Tx\|_A^2 \right) : x\in \mathcal{H}, \|x\|_A=1 \right\},\\
 dw_A(T)&= &\sup \left\{ \sqrt{|\langle Tx, x \rangle_A|^2+ \|Tx\|_A^4} : x\in \mathcal{H}, \|x\|_A=1 \right \}.
\end{eqnarray*} 
 It is easy to see that the $A$-Davis-Wielandt radius of $T\in \mathcal{B}_{A^{1/2}}(\mathcal{H})$ satisfying the following inequality:
\begin{eqnarray}\label{1st}
\max \{ w_A(T), \|T\|_A^2 \} \leq dw_A(T)\leq \sqrt{w^2_A(T)+\|T\|_A^4}.
\end{eqnarray}
\noindent Recently, Feki in \cite{F2} have obtained some upper bounds for the $A$-Davis-Wielandt radius of operators in $\mathcal{B}_{A}({\mathcal{H}}).$
 In this paper, we study about the equality of the lower bounds for the $A$-Davis-Wielandt radius of $A$-bounded operators mentioned  in (\ref{1st}). We obtain  upper and lower bounds for the $A$-Davis-Wielandt radius of operators in $\mathcal{B}_{A}({\mathcal{H}}),$ which generalize and improve on the existing ones. Further we obtain inequalities for the $\mathbb{A}$-Davis-Wielandt radius of $2 \times 2$ operator matrices in $\mathcal{B}_{\mathbb{A}}({\mathcal{H}\oplus \mathcal{H}}),$ which generalize  inequalities in \cite{BBBP}. Next, we obtain an upper bound for the $A$-Davis-Wielandt radius of sum of product operators in $\mathcal{B}_{A}({\mathcal{H}}),$ i.e., if $P,Q,X,Y \in \mathcal{B}_{A}(\mathcal{H})$ then for any $t \in \mathbb{R}\setminus\{0\}$, we have 
$$dw_{A}^2(PXQ^{\sharp_A}\pm QYP^{\sharp_A}) \leq (t^2\|P\|_{A}^2+\frac{1}{t^2}\|Q\|_{A}^2)^2 \{ (t^2\|PX\|_{A}^2+\frac{1}{t^2}\|QY\|_{A}^2)^2 +\alpha^2 \},$$ 
where $\alpha=w_{\mathbb{A}}\left(\begin{array}{cc}
	O & X\\
	Y & O
	\end{array}\right).$
Finally, we compute the exact value for the $\mathbb{A}$-Davis-Wielandt radius of two operator matrices $\left(\begin{array}{cc}
I & X\\
0 & 0
\end{array}\right)$ and $\left(\begin{array}{cc}
0 & X\\
0 & 0
\end{array}\right)$, where $X \in \mathcal{B}_{A^{1/2}}(\mathcal{H})$.

\section{ \textbf{Main results}}
\noindent We begin this section with the study of the equality conditions of both upper and lower bounds of $A$-bounded operators mentioned  in (\ref{1st}). Fisrt we mention  the following known result $($see \cite[Th. 11 and Prop. 4]{FA}$)$.

\begin{theorem}\label{uppereql}
Let $T\in \mathcal{B}_{A^{1/2}}({\mathcal{H}}).$ Then the following conditions are equivalent:\\
$(i)$ $dw_{A}(T)=\sqrt{w^2_A(T)+\|T\|_A^4}$.\\
$(ii)$ T is $A$-normaloid, i.e, $w_A(T)=\|T\|_A$.\\
$(iii)$ There exist a sequence of $A$-unit vectors $\{x_n\}$ in $\mathcal{H}$ such that $$\lim_{n\to \infty}\|Tx_n\|_A=\|T\|_A ~~\mbox{and}~~ \lim_{n\to \infty}|\langle Tx_n,x_n\rangle _A|=w_A(T).$$
\end{theorem}

\begin{remark}
If $\mathcal{H}$ is finite-dimensional then condition (iii) of Theorem \ref{uppereql} is replaced by $M_T^{A}\cap W_T^{A}\neq \emptyset$, i.e., there exist a $A$-unit vector $x$ in $\mathcal{H}$ such that $\|Tx\|_A=\|T\|_A$ and $|\langle Tx,x\rangle _A|=w_A(T).$ 
\end{remark}

\noindent Next we find the equality condition of the first inequality in (\ref{1st}).

\begin{theorem}
Let $T\in \mathcal{B}_{A^{1/2}}({\mathcal{H}}).$ Then the following conditions are equivalent:\\
$(i)$ $dw_{A}(T)=w_{A}(T)$.\\
$(ii)$ $AT=0$.
\end{theorem}

\begin{proof}
The  part $(ii) \Rightarrow (i)$  follows trivially.  We only prove $ (i) \Rightarrow (ii).$ Since $T\in \mathcal{B}_{A^{1/2}}({\mathcal{H}}),$ there exists a sequence $\{x_n\}$ in $\mathcal{H}$ with $\|x_n\|_{A}=1$ such that $w_{A}(T)=\lim_{n\to \infty} |\langle Tx_n,x_n\rangle_{A}|.$ The sequence $ \{ \|Tx_n\|_{A}\} ,$ being a bounded sequence of real numbers has a convergent subsequence $\{ \|Tx_{n_k}\|_{A} \}. $ Now $w_{A}^2(T)=dw_{A}^2(T)\geq |\langle Tx_{n_k},x_{n_k}\rangle_{A}|^2+\|Tx_{n_k}\|_{A}^4.$ Taking limit on both sides, we get $w_{A}^2(T)=dw_{A}^2(T)\geq w_{A}^2(T)+ \lim_{k\to \infty}\|Tx_{n_k}\|_{A}^4.$ This implies that $\lim_{k\to \infty}\|Tx_{n_k}\|_{A}=0$. Therefore, it follows from Cauchy-Schwarz inequality that $w_{A}(T)=\lim_{k\to \infty} |\langle Tx_{n_k},x_{n_k}\rangle_{A}|\leq \lim_{k\to \infty} \|Tx_{n_k}\|_{A}=0.$ So, we get $w_{A}(T)=0$ and hence, $AT=0.$
\end{proof}

\begin{theorem}\label{th-equality2}
Let $T\in \mathcal{B}_{A^{1/2}}({\mathcal{H}})$ and $dw_{A}(T)=\|T\|_{A}^2.$ Then either of the following condition holds:\\
$(i)$ Let $M^A_{T}\neq \emptyset $. Then $|\langle Tx,x\rangle_{A}|=0$ if $x\in M^A_{T} $, i.e., $M^A_{T} \subseteq c^A_{T}.$\\
$(ii)$ Let $M^A_{T}= \emptyset $. Then there exists a sequence $\{x_n\}$ in $\mathcal{H}$ with $\|x_n\|_{A}=1$ such that $\lim_{n\to \infty}\|Tx_n\|_{A}=\|T\|_{A}$ and $\lim_{n\to \infty} |\langle Tx_n,x_n\rangle_{A}|=0.$ 
\end{theorem}

\begin{proof}
$(i)$ Let $M^A_{T}\neq \emptyset $ and $x\in M^A_{T}.$ So, $\|Tx\|_{A}^4=\|T\|_{A}^4=dw_{A}^2(T)\geq |\langle Tx,x\rangle_{A}|^2+\|Tx\|_{A}^4$. This implies that $|\langle Tx,x\rangle_{A}|=0.$ so $x\in c^A_{T}.$ Therefore, $M^A_{T} \subseteq c^A_{T}.$ \\
 $(ii)$ Let $M^A_{T}= \emptyset $. Since $T\in \mathcal{B}_{A^{1/2}}({\mathcal{H}}),$ there exists a sequence $\{x_n\}$ in $\mathcal{H}$ with $\|x_n\|_{A}=1$ such that $\|T\|_{A}=\lim_{n\to \infty} \| Tx_n \|_{A}.$ Since $ \{ |\langle Tx_n,x_n\rangle_{A}| \} $ is a bounded sequence of scalars, so it has a convergent subsequence $ \{ |\langle Tx_{n_k},x_{n_k}\rangle_{A}| \} .$ Now $\|T\|_{A}^4 =dw_{A}^2(T)\geq |\langle Tx_{n_k},x_{n_k}\rangle_{A}|^2+\|Tx_{n_k}\|_{A}^4.$ Taking limit on both sides, we get $\|T\|_{A}^4=dw_{A}^2(T)\geq \lim_{k\to \infty}|\langle Tx_{n_k},x_{n_k}\rangle_{A}|^2+ \|T\|_{A}^4$ and so,  $\lim_{k\to \infty}|\langle Tx_{n_k},x_{n_k}\rangle_{A}|=0.$ This completes the proof.
\end{proof}

\begin{remark}
We note  that the converse part of Theorem \ref{th-equality2} may not hold, (see \cite[Remark 2.3]{BBBP}.
\end{remark}

\noindent We next obtain lower bounds for the $A$-Davis-Wielandt radius of operators in $\mathcal{B}_{A}(\mathcal{H})$.

\begin{theorem}\label{th-lower1}
Let $T\in \mathcal{B}_{A}({\mathcal{H}}).$ Then
\begin{eqnarray*}
(i)~ dw_{A}^2(T) & \geq & \max \left\{ w_{A}^2(T)+c_{A}^2(|T|^2_A), \|T\|_{A}^4+c_{A}^2(T)\right\},\\
(ii)~dw_{A}^2(T) & \geq & 2\max \left\{w_{A}(T)c_{A}(|T|^2_A), c_{A}(T)\|T\|_{A}^2 \right\}.
\end{eqnarray*}
\end{theorem}

\begin{proof}
$ (i) $ Let $x$ be a $A$-unit vector in $\mathcal{H}$. Then from the definition of $dw_{A}(T)$, we get
\begin{eqnarray*}
dw_{A}^2(T) &\geq & |\langle Tx,x \rangle_{A} |^2+\|Tx\|_{A}^4\\
 &= & |\langle Tx,x \rangle _{A}|^2+\langle |T|^2_Ax,x\rangle_{A} ^2 \\
&\geq & |\langle Tx,x \rangle _{A}|^2 + c_{A}^2(|T|^2_A).
\end{eqnarray*}
Therefore, taking supremum over all $A$-unit vectors in $\mathcal{H}$, we have
\[dw_{A}^2(T)\geq  w_{A}^2(T)+c_{A}^2(|T|^2_A).\]

Again from $dw_{A}^2(T) \geq  |\langle Tx,x \rangle _{A}|^2+\|Tx\|_{A}^4,$ where $\|x\|_{A}=1$, we get
\[dw_{A}^2(T) \geq  c_{A}^2(T)+\|Tx\|_{A}^4.\]
Taking supremum over all $A$-unit vectors in $\mathcal{H}$, we have
\[dw_{A}^2(T) \geq  c_{A}^2(T)+\|T\|_{A}^4.\]
This completes the proof of $(i).$\\
$(ii)$  	For all  $x\in \mathcal{H}$ with $\|x\|_{A}=1,$ we have  $$|\langle Tx,x \rangle _{A}|^2+\|Tx\|_{A}^4\geq 2|\langle Tx,x \rangle _{A}| \|Tx\|_{A}^2$$ and  so, 
$$ dw_{A}^2(T)\geq 2|\langle Tx,x \rangle _{A}| \langle |T|^2_Ax,x \rangle _{A}\geq 2|\langle Tx,x \rangle _{A}| c_{A}(|T|^2_A).$$

Taking supremum over all $A$-unit vectors in $\mathcal{H}$, we get 
\[dw_{A}^2(T)\geq 2w_{A}(T)c_{A}(|T|^2_A).\]

Again from $|\langle Tx,x \rangle _{A}|^2+\|Tx\|_{A}^4\geq 2|\langle Tx,x \rangle _{A}| \|Tx\|_{A}^2$, we have
\[dw_{A}^2(T)\geq 2c_{A}(T) \|Tx\|_{A}^2.\]
Taking supremum over all $A$-unit vectors in $\mathcal{H}$, we get 
\[dw_{A}^2(T)\geq 2c_{A}(T)\|T\|_{A}^2.\]
This completes the proof.
\end{proof}

\begin{remark}
(i) It is easy to observe that the lower bound of the $A$-Davis-Wielandt radius of $T\in \mathcal{B}_{A}({\mathcal{H}})$  obtained in Theorem \ref{th-lower1} (i) is sharper than that in (\ref{1st}). \\
(ii) Also, both the inequalities in \cite[Th. 2.4]{BBBP} follow from Theorem \ref{th-lower1} by considering $A=I$.
\end{remark}

\noindent In the following theorem we obtain an upper bound for the $A$-Davis-Wielandt radius of operators in $\mathcal{B}_{A}(\mathcal{H})$.

\begin{theorem}\label{th-upper2}
Let $T\in \mathcal{B}_{A}({\mathcal{H}}).$ Then 
\[dw_{A}^2(T)\leq \sup_{\theta \in \mathbb{R}}w_{A}^2(e^{{\rm i} \theta}T+|T|^2_A)-2c_{A}(T)m_{A}^2(T).\]
\end{theorem}

\begin{proof}
Let $x\in \mathcal{H}$ with $\|x\|_{A}=1.$ Then there exists $\theta \in \mathbb{R}$ such that $|\langle Tx,x \rangle _{A}|=e^{{\rm i} \theta}\langle Tx,x\rangle _{A}$. Now,
\begin{eqnarray*}
&& |\langle Tx,x \rangle _{A}|^2+\|Tx\|_{A}^4 \\
=&& \langle e^{{\rm i} \theta}Tx,x\rangle _{A}^2+ \langle |T|^2_Ax,x\rangle _{A} ^2 \\
=&& (\langle e^{{\rm i} \theta}Tx,x\rangle _{A}+ \langle |T|^2_Ax,x\rangle _{A})^2-2\langle e^{{\rm i} \theta}Tx,x\rangle _{A}\langle |T|^2_Ax,x\rangle _{A}.
\end{eqnarray*}
Hence, \begin{eqnarray*}
2\langle e^{{\rm i} \theta}Tx,x\rangle _{A}\langle |T|^2_Ax,x\rangle _{A}+|\langle Tx,x \rangle _{A}|^2+\|Tx\|_{A}^4 &=&\left (\langle e^{{\rm i} \theta}Tx,x\rangle _{A}+ \langle |T|^2_Ax,x\rangle _{A}\right)^2\\
\Rightarrow 2\langle e^{{\rm i} \theta}Tx,x\rangle _{A}\langle |T|^2_Ax,x\rangle _{A}+|\langle Tx,x \rangle _{A}|^2+\|Tx\|_{A}^4&=& \langle (e^{{\rm i} \theta}T+  |T|^2_A)x,x\rangle _{A}^2\\
\Rightarrow 2|\langle Tx,x\rangle_{A}| \langle |T|^2_Ax,x\rangle _{A}+|\langle Tx,x \rangle _{A}|^2+\|Tx\|_{A}^4&\leq& w_{A}^2(e^{{\rm i} \theta}T+ |T|^2_A).
\end{eqnarray*}
Therefore, $$2|\langle Tx,x\rangle_{A} |~~\langle |T|^2_Ax,x\rangle_{A} +|\langle Tx,x \rangle_{A} |^2+\|Tx\|_{A}^4 \leq \sup_{\theta \in \mathbb{R}} w_{A}^2(e^{{\rm i} \theta}T+ |T|^2_A)$$ and so,
\[2c_{A}(T)m_{A}^2(T)+|\langle Tx,x \rangle _{A}|^2+\|Tx\|_{A}^4 \leq \sup_{\theta \in \mathbb{R}} w_{A}^2(e^{{\rm i} \theta}T+ |T|^2_A).\]
Hence, taking supremum over all $A$-unit vectors in $\mathcal{H}$, we get 
\[2c_{A}(T)m_{A}^2(T)+dw_{A}^2(T) \leq \sup_{\theta \in \mathbb{R}} w_{A}^2(e^{{\rm i} \theta}T+ |T|^2_A).\]
\[ \Rightarrow dw_{A}^2(T)\leq \sup_{\theta \in \mathbb{R}}w_{A}^2(e^{{\rm i} \theta}T+|T|^2_A)-2c_{A}(T)m_{A}^2(T).\]
\end{proof}

\begin{remark}
We would like to note that the inequality in \cite[Th. 2.6]{BBBP} follows from Theorem \ref{th-upper2} by considering $A=I$.
\end{remark}

\noindent Next we obtain the following upper and lower bounds for the $A$-Davis-Wielandt radius of operators in $\mathcal{B}_{A}(\mathcal{H})$.

\begin{theorem}\label{th-upperlower4}
Let $T\in \mathcal{B}_{A}(\mathcal{H}).$  Then 
	\begin{eqnarray*}
\frac{1}{2} \left \{w_{A}^2(T+|T|^2_A)+ c_{A}^2(T-|T|^2_A)\right\}&\leq& dw_{A}^2(T) \\
&\leq & \frac{1}{2} \left \{w_{A}^2(T+|T|^2_A)+ w_{A}^2(T-|T|^2_A)\right \}.  
	\end{eqnarray*}
	\end{theorem}
	
	\begin{proof}
	Let $x\in \mathcal{H}$ with $\|x\|_{A}=1.$ Then
	\begin{eqnarray*}
	|\langle Tx,x \rangle _{A}|^2+\|Tx\|_{A}^4 &=& \frac{1}{2}\left |\langle Tx,x \rangle _{A}+\langle Tx, Tx\rangle _{A} \right|^2+ \frac{1}{2}\left|\langle Tx,x \rangle _{A}-\langle Tx, Tx\rangle _{A} \right|^2\\
	&=& \frac{1}{2}\left|\langle Tx,x \rangle _{A}+\langle |T|^2_Ax, x\rangle  _{A}\right|^2+ \frac{1}{2}\left|\langle Tx,x \rangle _{A}-\langle |T|^2_Ax, x\rangle _{A} \right|^2\\
	&=& \frac{1}{2}\left|\langle (T+ |T|^2_A)x, x\rangle  _{A}\right|^2+ \frac{1}{2}\left|\langle (T- |T|^2_A)x, x\rangle  _{A}\right|^2\\
	&\geq& \frac{1}{2} \left \{\left|\langle (T+ |T|^2_A)x, x\rangle  _{A}\right|^2+ c_{A}^2(T-|T|^2_A)\right\}.
\end{eqnarray*}
Therefore, taking supremum over all $A$-unit vectors in $\mathcal{H}$, we get
	\[dw_{A}^2(T) \geq  \frac{1}{2} \left \{w_{A}^2(T+|T|^2_A)+ c_{A}^2(T-|T|^2_A)\right\}.\]
Again,
\begin{eqnarray*}
	|\langle Tx,x \rangle _{A}|^2+\|Tx\|_{A}^4 &=& \frac{1}{2}\left |\langle Tx,x \rangle _{A}+\langle Tx, Tx\rangle _{A} \right|^2+ \frac{1}{2}\left|\langle Tx,x \rangle _{A}-\langle Tx, Tx\rangle _{A} \right|^2\\
	&=& \frac{1}{2}\left|\langle Tx,x \rangle _{A}+\langle |T|^2_Ax, x\rangle  _{A}\right|^2+ \frac{1}{2}\left|\langle Tx,x \rangle _{A}-\langle |T|^2_Ax, x\rangle _{A} \right|^2\\
	&=& \frac{1}{2}\left|\langle (T+ |T|^2_A)x, x\rangle  _{A}\right|^2+ \frac{1}{2}\left|\langle (T- |T|^2_A)x, x\rangle  _{A}\right|^2\\
	&\leq& \frac{1}{2} \left \{w_{A}^2(T+|T|^2_A)+ w_{A}^2(T-|T|^2_A)\right\}.
\end{eqnarray*}
Therefore, taking supremum over all $A$-unit vectors in $\mathcal{H}$, we get
	\[dw_{A}^2(T) \leq  \frac{1}{2} \left \{w_{A}^2(T+|T|^2_A)+ w_{A}^2(T-|T|^2_A)\right\}.\] Hence completes the proof.
\end{proof}

\begin{remark}
We would like to remark that the inequality obtained in Theorem \ref{th-upperlower4} is generalizes the inequality in \cite[Th. 2.8]{BBBP}.
\end{remark}

\noindent In the next theorem we obtain upper bounds for the $A$-Davis-Wielandt radius of $T\in \mathcal{B}_{A}({\mathcal{H}})$. First we need the following lemma.

\begin{lemma}\label{lem13}
Let $x,y,e \in \mathcal{H}$ with $\|e\|_{A}=1$. Then 
\[|\langle x,e\rangle _{A} \langle e,y \rangle_{A}|\leq \frac{1}{2}\left (|\langle x,y \rangle _{A}|+\|x\|_{A}\|y\|_{A} \right ).\]
\end{lemma}	

\begin{proof}
For all $a,b,c,d \in \mathbb{R},$ we have $(ac-bd)^2 \geq (a^2-b^2)(c^2-d^2)$. Using this and the Cauchy Schwarz inequality, we get
\begin{eqnarray*}
|\left \langle x-\langle x,e \rangle_{A}e,y -\langle y,e \rangle_{A}e \right \rangle _{A} |^2 &\leq& \|x-\langle x,e \rangle_{A}e\|_{A}^2\|y -\langle y,e \rangle_{A}e\|_{A}^2\\
\implies |\langle x,y \rangle_{A}-\langle x,e \rangle_{A}\langle e,y \rangle_{A}|^2 &\leq& (\|x\|_{A}^2-|\langle x,e \rangle_{A}|^2)(\|y\|_{A}^2-|\langle y,e \rangle_{A}|^2)\\
\implies |\langle x,y \rangle_{A}-\langle x,e \rangle_{A}\langle e,y \rangle_{A}|^2 &\leq& (\|x\|_{A}\|y\|_{A}-|\langle x,e \rangle_{A}||\langle y,e \rangle_{A}|)^2.
\end{eqnarray*}
\noindent Since $|\langle x,e \rangle_{A}| \leq \|x\|_{A}$ and $|\langle y,e \rangle_{A}| \leq \|y\|_{A}$, so $(\|x\|_{A}\|y\|_{A}-|\langle x,e \rangle_{A}||\langle y,e \rangle_{A}|) \geq 0.$ Therefore,  
\begin{eqnarray*}
|\langle x,y \rangle_{A}-\langle x,e \rangle_{A}\langle e,y \rangle_{A}| &\leq& \|x\|_{A}\|y\|_{A}-|\langle x,e \rangle_{A}||\langle y,e \rangle_{A}|\\
\implies |\langle x,e \rangle_{A}\langle e,y \rangle_{A}|-|\langle x,y \rangle_{A}| &\leq& \|x\|_{A}\|y\|_{A}-|\langle x,e \rangle_{A}||\langle e,y \rangle_{A}|.
\end{eqnarray*}
Hence, \[ 2|\langle x,e \rangle_{A}\langle e,y \rangle_{A}| \leq  |\langle x,y \rangle_{A}|+\|x\|_{A}\|y\|_{A}.\]
This completes the proof of the lemma.
\end{proof}
\begin{theorem}\label{upper15}
Let $T\in \mathcal{B}_{A}(\mathcal{H}).$ Then the following inequalities hold:
\begin{eqnarray*}
(i) && dw_{A}^2(T)\leq \left \||T|_{A}^2+{(|T|_{A}^2)}^{\sharp_A}|T|_{A}^2 \right \|_A,\\
(ii) && dw_{A}^2(T)\leq \frac{1}{2}\left(w_{A}(T^2)+\|T\|_{A}^2 \right)+\left \| T \right\|_{A}^4.
\end{eqnarray*}
\end{theorem}	
\begin{proof}
Let $x\in \mathcal{H}$ with $\|x\|_{A}=1$. Then  using Lemma \ref{lem13} we get, 
\begin{eqnarray*}
	|\langle Tx,x \rangle _{A}|^2+\|Tx\|_{A}^4 & = & |\langle Tx,x \rangle _{A} \langle x,Tx \rangle _{A}| + \langle |T|_{A}^2x,x \rangle _{A} \langle x,|T|_{A}^2x\rangle _{A}\\ 
	&\leq & \frac{1}{2}(\|Tx\|_{A}^2+ \langle Tx,Tx \rangle _{A}) + \frac{1}{2}(\||T|_{A}^2x\|_{A}^2+\langle |T|_{A}^2x,|T|_{A}^2x \rangle _{A}) \\
	&= & \langle |T|_{A}^2x,x \rangle _{A} +\langle {(|T|_{A}^2)}^{\sharp_A}|T|_{A}^2x,x \rangle _{A}\\
	&= & \langle (|T|_{A}^2+{(|T|_{A}^2)}^{\sharp_A}|T|_{A}^2)x,x \rangle_{A}.
\end{eqnarray*}
Therefore, taking supremum over all $A$-unit vectors in $\mathcal{H}$, we get the inequality (i). Again considering $|\langle Tx,x \rangle _{A}|^2 = |\langle Tx,x \rangle_{A} \langle x,T^{\sharp_A}x \rangle _{A}|$ and then using Lemma \ref{lem13}, we get the inequality (ii).
\end{proof}

\begin{remark}	
It is well-known that if $T$ is $A$-normaloid then $\|T^2\|_A=\|T\|_A^2$. Therefore, it is easy to observe that both the inequalities in Theorem \ref{upper15} becomes equality if $T$ is $A$-normaloid.
\end{remark}

\noindent In the next theorem we obtain an upper bound for the $A$-Davis-Wielandt radius of operators in $\mathcal{B}_A(\mathcal{H})$. For this we need the following lemma which  follows from Lemma \ref{lem13}. 
\begin{lemma}\label{lem14}
Let $x,y,e \in \mathcal{H}$ with $\|e\|_{A}=1$. Then
\[\|x\|_{A}^2\|y\|_{A}^2-|\langle x,y \rangle _{A}|^2 \geq 2|\langle x,e \rangle _{A}\langle e,y \rangle _{A}|\left (\|x\|_{A}\|y\|_{A}-|\langle x,y \rangle _{A}|\right).\]
\end{lemma}

\begin{theorem}\label{upper17}
Let $T\in \mathcal{B}_{A}(\mathcal{H}).$ Then 
\begin{eqnarray*}
dw_{A}^2(T)&\leq & 3\left \| {(|T|_{A}^2)}^{\sharp_A}|T|_{A}^2+|T|_{A}^2\right \|_A \\
&&  -c_{A}(|T|_{A}^2+T)m_{A}(|T|_{A}^2+T)-c_{A}(|T|_{A}^2-T)m_{A}(|T|_{A}^2-T).
\end{eqnarray*}
\end{theorem}
\begin{proof}
Let $x\in \mathcal{H}$ with $\|x\|_{A}=1$. Then using Lemma \ref{lem14} and Lemma \ref{lem13} we get,
\begin{eqnarray*}
|\langle Tx,x \rangle_{A}|^2 &\leq& \|Tx\|_{A}^2\|x\|_{A}^2-2|\langle Tx,x \rangle _{A}\langle x,x \rangle_{A}|(\|Tx\|_{A}\|x\|_{A}-|\langle Tx,x \rangle_{A}|) \\
&=& \|Tx\|_{A}^2+2|\langle Tx,x \rangle_{A}||\langle x,Tx \rangle_{A}|-2|\langle Tx,x \rangle_{A}|\|Tx\|_{A} \\
&\leq& \|Tx\|_{A}^2 +\|Tx\|_{A}^2+\langle Tx,Tx \rangle_{A}-2c_{A}(T)\|Tx\|_{A} \\
&\leq& 3\langle |T|_{A}^2x,x \rangle_{A}-2c_{A}(T)m_{A}(T).
\end{eqnarray*}
Using the above inequality, we get
\begin{eqnarray*}
	&& |\langle Tx,x \rangle _{A}|^2+\|Tx\|_{A}^4 \\
	=&& \frac{1}{2}\left (|\|Tx\|_{A}^2+\langle Tx,x \rangle_{A}|^2+|\|Tx\|_{A}^2-\langle Tx,x \rangle_{A}|^2\right ) \\
	=&&\frac{1}{2}\left (|\langle (|T|_{A}^2+T)x,x \rangle_{A}|^2+|\langle (|T|_{A}^2-T)x,x \rangle_{A}|^2 \right )\\
	\leq && \frac{1}{2}\Big(3 \left \langle \Big||T|_{A}^2+T\Big|_{A}^2x,x \right \rangle_{A}-2c_{A}(|T|_{A}^2+T)m_{A}(|T|_{A}^2+T) \\
	&& +3 \left \langle \Big||T|_{A}^2-T\Big|_{A}^2x,x \right \rangle_{A}-2c_{A}(|T|_{A}^2-T)m_{A}(|T|_{A}^2-T)\Big)\\
	=&& \frac{3}{2}\left \langle \left (\Big||T|_{A}^2+T\Big|_{A}^2+\Big||T|_{A}^2-T\Big|_{A}^2\right )x,x \right \rangle_{A} -c_{A}(|T|_{A}^2+T)m_{A}(|T|_{A}^2+T)\\
	&&-c_{A}(|T|_{A}^2-T)m_{A}(|T|_{A}^2-T)\\
	=&& 3\left \langle \left ({(|T|_{A}^2)}^{\sharp_A}|T|_{A}^2+|T|_{A}^2\right )x,x \right \rangle_{A}\\
	&& -c_{A}(|T|_{A}^2+T)m_{A}(|T|_{A}^2+T) -c_{A}(|T|_{A}^2-T)m_{A}(|T|_{A}^2-T).
\end{eqnarray*}
Therefore, taking supremum over all $A$-unit vectors in $\mathcal{H}$, we get the required inequality.
\end{proof}

\begin{remark}
We would like to note that the inequality in \cite[Th. 2.20]{BBBP} follows from Theorem \ref{upper17} by considering $A=I$.
\end{remark}

\noindent Next we prove the following lemma. 

\begin{lemma} \label{lem18}
Let $x,y \in \mathcal{H}$ and  $\lambda \in \mathbb{C}$. Then we have the following equality:
\[\|x\|_{A}^2\|y\|_{A}^2-|\langle x,y\rangle_{A}|^2=\|x-\lambda y\|_{A}^2\|y\|_{A}^2-|\langle x-\lambda y,y\rangle_A|^2.\]
\end{lemma}

\begin{proof}
We have, 
\begin{eqnarray*}
&& \|x-\lambda y\|_{A}^2\|y\|_{A}^2-|\langle x-\lambda y,y\rangle_{A}|^2 \\
=&& \langle x-\lambda y,x-\lambda y\rangle_{A}\|y\|_{A}^2-|\langle x,y\rangle_{A}-\lambda \|y\|_{A}^2 |^2 \\
=&& \left(\|x\|_{A}^2+|\lambda|^2\|y\|_{A}^2-2Re(\overline{\lambda}\langle x,y\rangle_{A})\right)\|y\|_{A}^2 -|\langle x,y\rangle_{A}|^2-|\lambda|^2\|y\|_{A}^4\\
&& +2 Re(\overline{\lambda}\langle x,y\rangle_{A})\|y\|_{A}^2 \\
=&& \|x\|_{A}^2\|y\|_{A}^2-|\langle x,y\rangle_{A}|^2.
\end{eqnarray*}
\end{proof}

\noindent Using Lemma \ref {lem18}, we obtain the following upper bound for the $A$-Davis-Wielandt radius of operator in $\mathcal{B}_{A}({\mathcal{H}}).$

\begin{theorem}\label{upper19}
Let $T\in \mathcal{B}_{A}(\mathcal{H}).$ Then
\begin{eqnarray*}
dw_{A}^2(T) &\leq& \inf_{\lambda \in \mathbb{R}}\sup_{\theta \in \mathbb{R}} \Big \{ 2|\lambda| \| \cos \theta Re_{A}(T)+|T|_{A}^2+\sin \theta Im_{A}(T)- \lambda I\|_{A} \\
&& +\frac{1}{2}	\|\cos \theta Re_{A}(T)+|T|_{A}^2+\sin \theta Im_{A}(T)-2\lambda I\|_{A}^2\\
&&+\frac{1}{2} \|\cos \theta Re_{A}(T)-|T|_{A}^2+\sin \theta Im_{A}(T)\|_{A}^2\Big\}.
\end{eqnarray*}
In particular,
\begin{eqnarray*}
dw_{A}^2(T) &\leq& \frac{1}{2}\sup_{\theta \in \mathbb{R}} \Big\{\left \|\cos \theta~ Re_{A}(T)+ |T|_{A}^2 +\sin \theta ~Im_{A}(T)\right \|_{A}^2\\
&& +\left \|\cos \theta~ Re_{A}(T)-|T|_{A}^2 +\sin \theta~ Im_{A}(T)\right \|_{A}^2 \Big\}.
\end{eqnarray*}
\end{theorem}

\begin{proof}
Let $x\in \mathcal{H}$ with $\|x\|_{A}=1.$ Then there exists $\theta \in \mathbb{R}$ such that $|\langle Tx,x \rangle _{A}|=e^{{-\rm i} \theta}\langle Tx,x\rangle_{A}$. Using   the Cartesian decomposition of $T$, i.e.,  $T=Re_{A}(T)+{\rm i}~Im_{A}(T),$ we get,
	\begin{eqnarray*}
	|\langle Tx,x \rangle_{A}| &=& \langle e^{{-\rm i} \theta}Tx,x \rangle_{A}\\
	&=& \langle ((\cos\theta-{\rm i} \sin\theta)(Re_{A}(T)+{\rm i}~ Im_{A}(T)))x,x \rangle_{A}\\
	&=& \langle (\cos\theta Re_{A}(T)+\sin\theta Im_{A}(T))x,x \rangle_{A} + {\rm i}\langle(\cos\theta Im_{A}(T)-\sin\theta Re_{A}(T) )x,x \rangle_{A}.
	\end{eqnarray*}
	Since $|\langle Tx,x \rangle_{A}|\in \mathbb{R} $,  $|\langle Tx,x \rangle_{A}|= \langle (\cos\theta Re_{A}(T)+\sin\theta Im_{A}(T))x,x \rangle_{A} $.
	Now using Lemma \ref{lem18}, we get  for any $\lambda \in \mathbb{R}$, 
	\begin{eqnarray*}
	|\langle Tx,x \rangle_{A}|^2 &=& |\langle (\cos\theta Re_{A}(T)+\sin\theta Im_{A}(T))x,x \rangle_{A}|^2\\
	&=& \|(\cos\theta Re_{A}(T)+\sin\theta Im_{A}(T))x\|_{A}^2\\
	&&-\|(\cos\theta Re_{A}(T)+\sin\theta Im_{A}(T))x-\lambda x\|_{A}^2\\
	&&+|\langle (\cos\theta Re_{A}(T)+\sin\theta Im_{A}(T))x-\lambda x,x \rangle _{A}|_{A}^2\\
	&=& \langle (\cos\theta Re_{A}(T)+\sin\theta Im_{A}(T))^2x,x \rangle _{A}\\
	&&-\langle(\cos\theta Re_{A}(T)+\sin\theta Im_{A}(T)-\lambda I)^2x,x\rangle_{A}\\
	&& +|\langle (\cos\theta Re_{A}(T)+\sin\theta Im_{A}(T)-\lambda I)x,x\rangle_{A} |^2\\
	&=& \Big \langle \Big \{(\cos\theta Re_{A}(T)+\sin\theta Im_{A}(T))^2 \\
	&& -(\cos\theta Re_{A}(T)+\sin\theta Im_{A}(T)-\lambda I)^2 \Big\}x,x \Big\rangle_{A}\\
	&& +|\langle(\cos\theta Re_{A}(T)+\sin\theta Im_{A}(T)-\lambda I)x,x\rangle_{A}|^2\\
	&=& \langle (2 \lambda (\cos\theta Re_{A}(T)+\sin\theta Im_{A}(T))-\lambda^2 I)x,x \rangle_{A}\\
	&&+ |\langle(\cos\theta Re_{A}(T)+\sin\theta Im_{A}(T)-\lambda I)x,x\rangle_{A}|^2.
	\end{eqnarray*}
	Similarly, using Lemma \ref{lem18}, we have 
	\begin{eqnarray*}
	\|Tx\|_{A}^4 &=& |\langle |T|_{A}^2x,x \rangle_{A}|^2 \\
	&=&\langle (2 \lambda |T|_{A}^2-{\lambda}^2 I)x,x \rangle _{A}+ |\langle(|T|_{A}^2-\lambda I)x,x \rangle_{A}|^2.
	\end{eqnarray*} 
Now,
\begin{eqnarray*}
|\langle Tx,x \rangle_{A}|^2+\|Tx\|_{A}^4 &=& \langle 2\lambda\{\cos \theta Re_{A}(T)+|T|_{A}^2+\sin \theta Im_{A}(T)\}x,x\rangle_{A}- 2\lambda^2\\
&& +\frac{1}{2}	|\langle(\cos \theta Re_{A}(T)+|T|_{A}^2+\sin \theta Im_{A}(T)-2\lambda I)x,x\rangle_{A}|^2\\
&&+\frac{1}{2}|\langle(\cos \theta Re_{A}(T)-|T|_{A}^2+\sin \theta Im_{A}(T))x,x\rangle_{A}|^2\\
&\leq& 2|\lambda| \| \cos \theta Re_{A}(T)+|T|_{A}^2+\sin \theta Im_{A}(T)- \lambda I\|_{A} \\
&& +\frac{1}{2}	\|\cos \theta Re_{A}(T)+|T|_{A}^2+\sin \theta Im_{A}(T)-2\lambda I\|_{A}^2\\
&&+\frac{1}{2} \|\cos \theta Re_{A}(T)-|T|_{A}^2+\sin \theta Im_{A}(T)\|_{A}^2\\
&\leq&  \sup_{\theta \in \mathbb{R}} \Big \{ 2|\lambda| \| \cos \theta Re_{A}(T)+|T|_{A}^2+\sin \theta Im_{A}(T)- \lambda I\|_{A} \\
&& +\frac{1}{2}	\|\cos \theta Re_{A}(T)+|T|_{A}^2+\sin \theta Im_{A}(T)-2\lambda I\|_{A}^2\\
&&+\frac{1}{2} \|\cos \theta Re_{A}(T)-|T|_{A}^2+\sin \theta Im_{A}(T)\|_{A}^2\Big\}.
\end{eqnarray*}
Therefore, taking supremum over all $A$-unit vectors in $\mathcal{H}$, we get 
\begin{eqnarray*}
dw_{A}^2(T) &\leq& \sup_{\theta \in \mathbb{R}} \Big \{ 2|\lambda| \| \cos \theta Re_{A}(T)+|T|_{A}^2+\sin \theta Im_{A}(T)- \lambda I\|_{A} \\
&& +\frac{1}{2}	\|\cos \theta Re_{A}(T)+|T|_{A}^2+\sin \theta Im_{A}(T)-2\lambda I\|_{A}^2\\
&&+\frac{1}{2} \|\cos \theta Re_{A}(T)-|T|_{A}^2+\sin \theta Im_{A}(T)\|_{A}^2\Big\}.
\end{eqnarray*}
This inequality holds for all $\lambda \in \mathbb{R}$, so we get the desired inequality.
In particular, if we choose $\lambda =0$, then
\begin{eqnarray*}
dw_{A}^2(T) &\leq&  \frac{1}{2}\sup_{\theta\in \mathbb{R}} \Big\{\left \|\cos \theta~ Re_{A}(T)+|T|_{A}^2 +\sin \theta ~Im_{A}(T)\right \|_{A}^2\\
&& +\left \|\cos \theta~ Re_{A}(T)-|T|_{A}^2 +\sin \theta~ Im_{A}(T)\right \|_{A}^2 \Big\}. 
\end{eqnarray*}
\end{proof}

\begin{remark}
We would like to note that the inequality in \cite[Th. 2.23]{BBBP} follows from Theorem \ref{upper19} by considering $A=I$.
\end{remark}

\noindent Next we obtain the following inequality.

\begin{theorem}\label{upper20}
Let $T\in \mathcal{B}_{A}(\mathcal{H}).$ Then 
\begin{eqnarray*}
 dw_{A}^2(T) &\leq& \inf_{\lambda \in \mathbb{C}}\Big \{\Big (2 \left \|Re(\lambda)~Re_{A}(T)+Im(\lambda)~Im_{A}(T)\right \|_{A}   +\left \||T|^2_A-2Re(\overline{\lambda}T)\right \|_{A}\Big )^2 \\
 && + 2 \|Re(\overline{\lambda}T)\|_{A}- |\lambda|^2+w_{A}^2(T-\lambda I) \Big \}.
\end{eqnarray*}
In particular, $dw_{A}(T) \leq \sqrt{w_{A}^2(T) + \|T\|_{A}^4}$.
\end{theorem}

\begin{proof}
Let $x\in \mathcal{H}$ with $\|x\|_{A}=1.$ Let $\lambda \in \mathbb{C}$. Using Lemma \ref{lem18}  we get, 
\begin{eqnarray*}
\|Tx\|_{A}^2\|x\|_{A}^2-|\langle Tx,x \rangle_{A}|^2 &=& \|Tx-\lambda x\|_{A}^2\|x\|_{A}^2-|\langle Tx-\lambda x ,x \rangle_{A}|^2.
\end{eqnarray*}
Using Cartesian decomposition of $T$, i.e.,  $T=Re_{A}(T)+{\rm i}~Im_{A}(T)$, we get,
\begin{eqnarray*}
 \|Tx\|_{A}^2 &=& \left (\langle Re_{A}(T)x,x \rangle_{A} \right)^2 -\left (\langle Re_{A}(T-\lambda I)x,x \rangle _{A}\right )^2+ \left (\langle Im_{A}(T)x,x \rangle _{A}\right)^2\\
&& - \left(\langle Im_{A}(T-\lambda I)x,x \rangle _{A}\right)^2+\|Tx-\lambda x\|_{A}^2\\
&=& \langle (2Re_{A}(T)-Re(\lambda)I)x,x\rangle_{A} \langle Re(\lambda) x,x \rangle_{A} \\
&&+\langle ( 2Im_{A}(T)-Im(\lambda)I)x,x\rangle \langle Im(\lambda) x,x \rangle _{A}+\|Tx-\lambda x\|_{A}^2\\
&=& 2Re(\lambda) \langle Re_{A}(T)x,x \rangle_{A} +2Im(\lambda) \langle Im_{A}(T)x,x \rangle _{A}\\
&&- (Re(\lambda))^2-(Im(\lambda))^2+\|Tx-\lambda x\|_{A}^2\\
&=& 2 \left (Re(\lambda)\langle Re_{A}(T)x,x \rangle+Im(\lambda) \langle Im_{A}(T)x,x \rangle _{A}\right)-|\lambda|^2\\
&& + \left \langle Tx-\lambda x ,Tx-\lambda x \right \rangle_{A}\\
&=& 2 \left (Re(\lambda)\langle Re_{A}(T)x,x \rangle_{A}+Im(\lambda) \langle Im_{A}(T)x,x \rangle_{A} \right)\\
&& + \left \langle (|T|^2_A-2Re_{A}(\overline{\lambda}T))x,x \right \rangle_{A}\\
&\leq& 2 \left \|Re(\lambda)~Re_{A}(T)+Im(\lambda)~Im_{A}(T)\right \|_{A}   +\left \||T|^2_A-2Re_{A}(\overline{\lambda}T)\right \|_{A}.
\end{eqnarray*}
Again using Lemma \ref{lem18} we get,
 \begin{eqnarray*}
|\langle Tx,x \rangle_{A}|^2 &=& \|Tx\|_{A}^2-\|Tx-\lambda x\|_{A}^2+|\langle Tx-\lambda x ,x \rangle_{A}|^2\\
&=& 2\langle Re(\overline{\lambda}T)x,x \rangle_{A} - |\lambda|^2+|\langle Tx-\lambda x ,x \rangle_{A}|^2\\
&\leq& 2 \|Re_{A}(\overline{\lambda}T)\|- |\lambda|^2+w_{A}^2(T-\lambda I).
\end{eqnarray*}
Hence, 
\begin{eqnarray*}
&& |\langle Tx,x \rangle_{A}|^2+\|Tx\|_{A}^4\\ 
&&\leq 2 \|Re_{A}(\overline{\lambda}T)\|- |\lambda|^2+w_{A}^2(T-\lambda I)\\
&&+  \left (2 \left \|Re(\lambda)~Re_{A}(T)+Im(\lambda)~Im_{A}(T)\right \|   +\left \||T|^2_A-2Re_{A}(\overline{\lambda}T)\right \|_{A} \right )^2.
\end{eqnarray*}
Therefore, taking supremum over all $A$-unit vectors in $\mathcal{H}$, and then taking infimum over all $\lambda \in \mathbb{C}$, we get
\begin{eqnarray*}
 dw_{A}^2(T) &\leq& \inf_{\lambda \in \mathbb{C}}\Big \{\left (2 \left \|Re(\lambda)~Re_{A}(T)+Im(\lambda)~Im_{A}(T)\right \|_{A}   +\left \||T|^2_A-2Re_{A}(\overline{\lambda}T)\right \|_{A}\right )^2 \\
 && + 2 \|Re_{A}(\overline{\lambda}T)\|_{A}- |\lambda|^2+w_{A}^2(T-\lambda I) \Big \}.
\end{eqnarray*}
Taking $\lambda =0$, we get $dw_{A}(T) \leq \sqrt{w_{A}^2(T) + \|T\|_{A}^4}.$
\end{proof}

\begin{remark}
We would like to note that the inequality in \cite[Th. 2.24]{BBBP} follows from Theorem \ref{upper20} by considering $A=I$.
\end{remark}

\noindent In the following theorem we obtain an upper bound for the $A$-Davis-Wielandt radius of sum of two operators in $ \mathcal{B}_{A}({\mathcal{H}}).$

\begin{theorem}\label{theorem24}
Let $X,Y \in \mathcal{B}_{A}(\mathcal{H})$. Then
\[dw_{A}(X+Y)\leq dw_{A}(X)+dw_{A}(Y)+w_{A}(X^{\sharp_A}Y+Y^{\sharp_A}X).\]
In particular, if $A(X^{\sharp_A}Y+Y^{\sharp_A}X)=O$ then $$dw_{A}(X+Y)\leq dw_{A}(X)+dw_{A}(Y).$$
\end{theorem}

\begin{proof}
From the definition of  the A-Davis-Wielandt shell we get,
\begin{eqnarray*}
DW_{A}(X+Y) &=& \left \{ \Big (\left \langle (X+Y)x,x \right \rangle _{A} , \left \langle (X+Y)x,(X+Y)x \right \rangle _{A}\Big) : x \in \mathcal{H},\|x\|_{A}=1 \right \}\\
&=& \Big\{ \Big( \langle Xx,x \rangle_{A} , \langle Xx,Xx \rangle_{A} \Big)+ \Big( \langle Yx,x \rangle _{A}, \langle Yx,Yx \rangle _{A}\Big)\\
&& +\Big(0,\langle (X^{\sharp_A}Y+Y^{\sharp_A}X)x ,x\rangle _{A}\Big) : x \in \mathcal{H},\|x\|_{A}=1 \Big \}.
\end {eqnarray*}
Hence, $DW_{A}(X+Y)\subseteq DW_{A}(X)+DW_{A}(Y)+A$, where
\[A = \left \{ \left (0,\langle (X^{\sharp_A}Y+Y^{\sharp_A}X)x ,x\rangle _{A}\right): x \in \mathcal{H},\|x\|_{A}=1 \right \}. \]
This implies the first inequality of the theorem. In particular, if we consider $A(X^{\sharp_A}Y+Y^{\sharp_A}X)=O,$ then we get the second inequality.
\end{proof}

\begin{remark}
If we consider $A=I$ in Theorem \ref{theorem24} then we get the inequalities in \cite[Lemma 3.3 and Prop. 3.4]{BBBP}.
\end{remark}

\noindent Next we state the following lemma, proof of which can be found in \cite[ Lemma 3.1]{BFP}.
\begin{lemma}\label{lem3.1}
Let $T_{ij}\in \mathcal{B}_{A}({\mathcal{H}})$, for $i,j=1,2.$ Then $(T_{ij})_{2\times 2}\in \mathcal{B}_{\mathbb{A}}({\mathcal{H}\oplus \mathcal{H}})$ and $$\left(\begin{array}{cc}
	T_{11} & T_{12}\\
	T_{21} & T_{22}
	\end{array}\right)^{\sharp_{\mathbb{A}}}=\left(\begin{array}{cc}
	T^{\sharp_A}_{11} & T^{\sharp_A}_{21}\\
	T^{\sharp_A}_{12} & T^{\sharp_A}_{22}
	\end{array}\right).$$
\end{lemma}

\noindent Using Theorem \ref{theorem24} and Lemma \ref{lem3.1}, we prove the following inequality.

\begin{cor}\label{upper26}
Let $X,Y \in \mathcal{B}_{A}(\mathcal{H})$, then
\[ dw_{\mathbb{A}}\left(\begin{array}{cc}
	O & X\\
	Y & O
	\end{array}\right) \leq \sqrt{\frac{1}{4}\|X\|_{A}^2+\|X\|_{A}^4}+\sqrt{\frac{1}{4}\|Y\|_{A}^2+\|Y\|_{A}^4}.\]
\end{cor}

\begin{proof}
Clearly, $\left(\begin{array}{cc}
	O & X\\
	O & O
	\end{array}\right)^{\sharp_{\mathbb{A}}}\left(\begin{array}{cc}
	O & O\\
	Y & O
	\end{array}\right)+\left(\begin{array}{cc}
	O & O\\
	Y & O
	\end{array}\right)^{\sharp_{\mathbb{A}}}\left(\begin{array}{cc}
	O & X\\
	O & O
	\end{array}\right)=\left(\begin{array}{cc}
	O & O\\
	O & O
	\end{array}\right)$. Therefore, from Theorem \ref{theorem24}, we get, 
	\begin{eqnarray*}
	&& dw_{\mathbb{A}}\left(\begin{array}{cc}
	O & X\\
	Y & O
	\end{array}\right) \\
	&\leq& dw_{\mathbb{A}}\left(\begin{array}{cc}
	O & X\\
	O & O
	\end{array}\right)+dw_{\mathbb{A}}\left(\begin{array}{cc}
	O & O\\
	Y & O
	\end{array}\right)\\
	&\leq& \sqrt{w_{\mathbb{A}}^2\left(\begin{array}{cc}
	O & X\\
	O & O
	\end{array}\right) +\left \|\left(\begin{array}{cc}
	O & X\\
	O & O
	\end{array}\right) \right \|_{\mathbb{A}}^4} +\sqrt{w_{\mathbb{A}}^2\left(\begin{array}{cc}
	O & O\\
	Y & O
	\end{array}\right)+\left \|\left(\begin{array}{cc}
	O & O\\
	Y & O
	\end{array}\right) \right\|_{\mathbb{A}}^4}\\
	&=& \sqrt{\frac{1}{4} \left \|\left(\begin{array}{cc}
	O & X\\
	O & O
	\end{array}\right) \right \|_{\mathbb{A}}^2+\left \|\left(\begin{array}{cc}
	O & X\\
	O & O
	\end{array}\right) \right \|_{\mathbb{A}}^4} +\sqrt{\frac{1}{4} \left \|\left(\begin{array}{cc}
	O & O\\
	Y & O
	\end{array}\right)\right \|_{\mathbb{A}}^2+\left \|\left(\begin{array}{cc}
	O & O\\
	Y & O
	\end{array}\right) \right\|_{\mathbb{A}}^4},\\
	&& \mbox{as}~~\mathbb{A}\left(\begin{array}{cc}
	O & X\\
	O & O
	\end{array}\right)^2=\mathbb{A}\left(\begin{array}{cc}
	O & O\\
	Y & O
	\end{array}\right)^2=\left(\begin{array}{cc}
	O & O\\
	O & O
	\end{array}\right),~~ \mbox {see  \cite[Cor. 2.2]{F}}\\
	&=& \sqrt{\frac{1}{4}\|X\|_{A}^2+\|X\|_{A}^4}+\sqrt{\frac{1}{4}\|Y\|_{A}^2+\|Y\|_{A}^4}, ~~\textit {by using \cite[Remark 3]{BNP}}.
	\end{eqnarray*}
\end{proof}

\begin{remark}
In particular, if we consider $A=I$ in Corollary \ref{upper26} then we have the inequality in \cite[Th. 3.5]{BBBP}.
\end{remark}

\noindent Next we state the following lemma, proof of which follows from $DW_A(U^{{\sharp_A}}TU)=DW_A(T)$, where $T \in \mathcal{B}_{A^{1/2}}(\mathcal{H})$ and $U\in \mathcal{B}_{A}(\mathcal{H})$  is an  $A$-unitary operator.%, (see in \cite[Prop. 1]{FA}).

	\begin{lemma}\label{lem22}
	Let $T \in \mathcal{B}_{A^{1/2}}(\mathcal{H})$ and $U\in \mathcal{B}_{A}(\mathcal{H})$ be  an $A$-unitary operator. Then
	$$dw_{A}(U^{{\sharp_A}}TU)=dw_{A}(T).$$
	\end{lemma}

	\noindent Using Lemma \ref{lem22}, we prove the following lemma.
	
\begin{lemma}\label{lem23}
	Let $X,Y \in \mathcal{B}_{A^{1/2}}(\mathcal{H})$. Then 
\begin{enumerate}[label=(\alph*)]
\item $dw_{\mathbb{A}}\left(\begin{array}{cc}
	O & X\\
	e^{{\rm i} \theta}Y & O
	\end{array}\right)=dw_{\mathbb{A}}\left(\begin{array}{cc}
	O & X\\
	Y & O
	\end{array}\right)$, for every $\theta \in \mathbb{R}.$
	\item $dw_{\mathbb{A}}\left(\begin{array}{cc}
	O & X\\
	Y & O
	\end{array}\right)=dw_{\mathbb{A}}\left(\begin{array}{cc}
	O & Y\\
	X & O
	\end{array}\right).$ 
	\end{enumerate}
	\end{lemma}

\begin{proof}
\begin{enumerate}[label=(\alph*)]
\item Let $U=\left(\begin{array}{cc}
	I & O\\
	O & e^{{\rm i}\frac{\theta}{2}}I
	\end{array}\right)$. Then  using Lemma \ref{lem22} we get, \\
	$dw_{\mathbb{A}}\left(\begin{array}{cc}
	O & X\\
	e^{{\rm i} \theta}Y & O
	\end{array}\right)=dw_{\mathbb{A}}\left(U^{\sharp_{\mathbb{A}}} \left(\begin{array}{cc}
	O & X\\
	e^{{\rm i}\theta}Y & O
	\end{array} \right) U\right)=dw_{\mathbb{A}}\left(\begin{array}{cc}
	O & e^{{\rm i}\frac{\theta}{2}} X\\
	e^{{\rm i}\frac{\theta}{2}} Y & O
	\end{array}\right) $ $= dw_{\mathbb{A}}\left(\begin{array}{cc}
	O & X\\
	Y & O
	\end{array}\right).$
	
\item  Considering  $U=\left(\begin{array}{cc}
	O & I\\
	I & O
	\end{array}\right)$  and  using Lemma \ref{lem22}, we get (b).
\end{enumerate}
\end{proof}

\noindent Using Lemma \ref{lem23}, we obtain an upper bound for the $A$-Davis-Wielandt radius of sum of product operators in $ \mathcal{B}_{A}({\mathcal{H}}).$

\begin{theorem}\label{th-30pro}
Let $P,Q,X,Y \in \mathcal{B}_{A}(\mathcal{H}).$ Then for any $t \in \mathbb{R}\setminus\{0\}$, we have 
$$dw_{A}^2(PXQ^{\sharp_A}\pm QYP^{\sharp_A}) \leq \left(t^2\|P\|_{A}^2+\frac{1}{t^2}\|Q\|_{A}^2\right)^2\left \{\left (t^2\|PX\|_{A}^2 +\frac{1}{t^2}\|QY\|_{A}^2\right )^2 +\alpha^2\right \},$$ 
where $\alpha=w_{\mathbb{A}}\left(\begin{array}{cc}
	O & X\\
	Y & O
	\end{array}\right).$
\end{theorem}
\begin{proof}
Let $C,Z \in \mathcal{B}_{\mathbb{A}}(\mathcal{H} \oplus \mathcal{H})$ be such that $C=\left(\begin{array}{cc}
	P & Q\\
	O & O
	\end{array}\right)$ and $Z=\left(\begin{array}{cc}
	O & X\\
	Y & O
	\end{array}\right)$. Then we have, $CZC^{\sharp_{\mathbb{A}}}=\left(\begin{array}{cc}
	PXQ^{\sharp_A}+QYP^{\sharp_A} & O\\
	O & O
	\end{array}\right).$ Therefore,
	\begin{eqnarray*}
	dw_{A}^2(PXQ^{\sharp_A}+QYP^{\sharp_A}) &\leq& dw_{\mathbb{A}}^2 \left(\begin{array}{cc}
	PXQ^{\sharp_A}+QYP^{\sharp_A} & O\\
	O & O
	\end{array}\right) \\
	&=& dw_{\mathbb{A}}^2(CZC^{\sharp_{\mathbb{A}}})\\
	&=& \sup_{\|x\|_{\mathbb{A}}=1}\left \{|\langle CZC^{\sharp_{\mathbb{A}}}x,x \rangle _{\mathbb{A}}|^2+\|CZC^{\sharp_{\mathbb{A}}}x\|_{\mathbb{A}}^4 \right \}\\
	&=& \sup_{\|x\|_{\mathbb{A}}=1}\left \{|\langle ZC^{\sharp_{\mathbb{A}}}x,C^{\sharp_{\mathbb{A}}}x \rangle _{\mathbb{A}}|^2+\|CZC^{\sharp_{\mathbb{A}}}x\|_{\mathbb{A}}^4 \right \}\\
	&\leq& \sup_{\|x\|_{\mathbb{A}}=1}\left \{w_{\mathbb{A}}^2(Z)\|C^{\sharp_{\mathbb{A}}}x\|_{\mathbb{A}}^4+\|CZ\|_{\mathbb{A}}^4\|C^{\sharp_{\mathbb{A}}}x\|_{\mathbb{A}}^4 \right \}\\
	&=& \left (w_{\mathbb{A}}^2(Z)+\|CZ\|_{\mathbb{A}}^4 \right)\|C\|_{\mathbb{A}}^4.
\end{eqnarray*}
It is easy to see that $\|C\|_{\mathbb{A}}^2=\|PP^{\sharp_A}+QQ^{\sharp_A}\|_{A}$ and $\|CZ\|_{\mathbb{A}}^2=\|(QY)(QY)^{\sharp_A}+(PX)(PX)^{\sharp_A}\|_{A}.$ Therefore, from the above inequality, we get
	$$dw_{A}^2(PXQ^{\sharp_A}+QYP^{\sharp_A}) \leq \left(\|P\|_{A}^2+\|Q\|_{A}^2\right)^2\left \{(\|QY\|_{A}^2+\|PX\|_{A}^2)^2 +w_{\mathbb{A}}^2\left(Z\right)\right \}.$$
Replacing $Y$ by $-Y$ in the above inequality and using Lemma \ref{lem23} (a), we get
$$dw_{A}^2(PXQ^{\sharp_A}-QYP^{\sharp_A}) \leq \left(\|P\|_{A}^2+\|Q\|_{A}^2\right)^2\left \{(\|QY\|_{A}^2+\|PX\|_{A}^2)^2 +w_{\mathbb{A}}^2\left(Z\right)\right \}.$$
Clearly, the above two inequalities hold for all $P,Q \in \mathcal{B}_{A}(\mathcal{H}).$ So, replacing  $P$ by $tP$ and $Q$ by $\frac{1}{t}Q$, we get the required inequality of the theorem.
\end{proof}

\begin{cor}\label{cor30b}
Let $P,Q,X,Y \in \mathcal{B}_{A}(\mathcal{H})$ with $\|P\|_{A},\|Q\|_{A} \neq 0.$	 then
\begin{eqnarray*}
 (i)~ dw_{A}^2(PXQ^{\sharp_A} \pm QYP^{\sharp_A}) \leq 4\|P\|_{A}^2\|Q\|_{A}^2\Big \{\Big (\frac{\|P\|_{A}}{\|Q\|_{A}}\|QY\|_{A}^2+\frac{\|Q\|_{A}}{\|P\|_{A}}\|PX\|_{A}^2 \Big )^2 +\alpha^2\Big \},
\end{eqnarray*}
where $\alpha=w_{\mathbb{A}}\left(\begin{array}{cc}
	O & X\\
	Y & O
	\end{array}\right).$ 
	\begin{eqnarray*}
(ii)~ dw_{A}^2(X \pm Y) ~~~~~~~~&\leq&  ~~~~~~~~~~~4\left \{\left (\|X\|_{A}^2+\|Y\|_{A}^2 \right )^2 +w_{\mathbb{A}}^2\left(\begin{array}{cc}
	O & X\\
	Y & O
	\end{array}\right)\right \}.\hspace{2.2cm}
	\end{eqnarray*}

\end{cor}

\begin{proof}
Considering $t=\sqrt{\frac{\|Q\|_{A}}{\|P\|_{A}}}$ in Theorem \ref{th-30pro}, we get the inequality (i). Choosing $P=Q=I$ in (i), we get the inequality (ii). 
\end{proof}

\begin{cor}\label{cor30c}
Let $P,Q,X,Y \in \mathcal{B}_{A}(\mathcal{H})$ be such that $\|PX\|_{A},\|QY\|_{A} \neq 0. $ Then 
\begin{eqnarray*}
(i) ~ dw_{A}^2(PXQ^{\sharp_A}\pm QYP^{\sharp_A}) \leq \left (\frac{\|QY\|_{A}}{\|PX\|_{A}}\|P\|_{A}^2+\frac{\|PX\|_{A}}{\|QY\|_{A}}\|Q\|_{A}^2 \right)^2\left \{4\|PX\|^2_{A}\|QY\|_{A}^2+\alpha^2\right \},
\end{eqnarray*}
where $\alpha=w_{\mathbb{A}}\left(\begin{array}{cc}
	O & X\\
	Y & O
	\end{array}\right).$ 
\begin{eqnarray*}
(ii)~ dw_{A}^2(X\pm Y) &\leq& \left (\frac{\|Y\|_{A}}{\|X\|_{A}} +\frac{\|X\|_{A}}{\|Y\|_{A}} \right )^2 \left\{\left(2\|X\|_{A}\|Y\|_{A} \right )^2+w_{\mathbb{A}}^2\left(\begin{array}{cc}
	O & X\\
	Y & O
	\end{array}\right) \right\}.
		\end{eqnarray*}
	\end{cor}
\begin{proof}
Considering $t=\sqrt{\frac{\|QY\|_{A}}{\|PX\|_{A}}}$ in Theorem \ref{th-30pro}, we get the inequality (i). Choosing $P=Q=I$ in (i), we get the inequality (ii). 
\end{proof}

\begin{remark}
Feki in \cite[Prop. 3]{FA} proved that if $X,Y \in \mathcal{B}_{A^{1/2}}(\mathcal{H})$ then the following inequality holds:
\[dw_A^2(X+Y)\leq 2\Big(dw_A(X)+dw_A(Y)\Big)+4\Big(dw_A(X)+dw_A(Y)\Big)^2.\] If we consider $A=\left(\begin{array}{cc}
	1 & 0\\
	0 & 2
	\end{array}\right)$, $X= \left(\begin{array}{cc}
	0 & 1\\
	0 & 0
	\end{array}\right)$ and $Y=\left(\begin{array}{cc}
	1 & 0\\
	0 & 0
	\end{array}\right)$ then \cite[Prop. 3]{FA} gives $dw_A(X+Y)\leq 4.2994$, whereas Theorem \ref{theorem24} gives $dw_A(X+Y)\leq 2.621320$, Corollary \ref{cor30b} (ii) gives $dw_A(X+Y)\leq 3.240466$ and Corollary \ref{cor30c} (ii) gives $dw_A(X+Y)\leq 3.26928$. Thus the bounds obtained in Theorem \ref{theorem24}, Corollary \ref{cor30b} (ii) and  Corollary \ref{cor30c} (ii) are better than that obtained in \cite[Prop. 3]{FA}.
\end{remark}

\noindent Now we determine the exact value of  the $\mathbb{A}$-Davis-Wielandt radius of  special type of $2\times 2$ operator matrices in  $\mathcal{B}_{A^{1/2}}(\mathcal{H}\oplus \mathcal{H})$.

\begin{theorem}\label{31}
Let $X \in \mathcal{B}_{A^{1/2}}(\mathcal{H})$ and  $\mathbb{T} =\left(\begin{array}{cc}
	I & X\\
	O & O
	\end{array}\right).$ Then 
	
		\[ dw_{\mathbb{A}}(\mathbb{T})=\begin{cases}
	\sqrt{2}, &\|X\|_{A}= 0 \\
	(\cos \theta_0 + \|X\|_{A} \sin \theta_0)(\cos^2 \theta_0 +(\cos \theta_0 + \|X\|_{A}\sin \theta_0)^2)^\frac{1}{2}, & \|X\|_{A}\neq 0,\\
	\end{cases} \]	
where $b=\|X\|_{A}$, $p=-\frac{2b^2-5}{2b},$ $q =- \frac{2b^2-2}{b^2},$ $r= -\frac{3}{2b},$ $s= \frac{1}{2^43^3b^6}(8b^8+20b^6+45b^4+61b^2+28),$ $\alpha=\frac{1}{27}(2p^3-9pq+27r),$ $\beta =(-\frac{\alpha}{2}+\sqrt{s})^\frac{1}{3},$ $\gamma=(-\frac{\alpha}{2}-\sqrt{s})^\frac{1}{3}$ and $\theta_0 = \tan^{-1}(\beta+\gamma-\frac{p}{3}).$

\end{theorem}

\begin{proof}
Let  $z=\left(\begin{array}{cc}
	x \\
	y 
	\end{array}\right) \in \mathcal{H}\oplus \mathcal{H} $ be such that $\|z\|_{\mathbb{A}}=1$, i.e, $\|x\|_{A}^2+\|y\|_{A}^2=1$.
	Then $ \langle \mathbb{T}z,z \rangle _{\mathbb{A}} = \langle x+Xy,x \rangle _{A} ~~\mbox{and} ~~ \langle \mathbb{T}z,\mathbb{T}z \rangle_{\mathbb{A}} = \langle x+Xy,x+Xy \rangle_{A}. $
	Now, we have
	\begin{eqnarray*}
	&& |\langle \mathbb{T}z,z \rangle_{\mathbb{A}}|^2+|\langle \mathbb{T}z,\mathbb{T}z \rangle_{\mathbb{A}}|^2\\
	&\leq& \|x+Xy\|_{A}^2\|x\|_{A}^2 +\|x+Xy\|_{A}^4 \\
	&=& \|x+Xy\|_{A}^2 \left (\|x\|_{A}^2 +\|x+Xy\|_{A}^2 \right)\\
	&\leq& \sup_{\|x\|_{A}^2+\|y\|_{A}^2=1}(\|x\|_{A}+\|X\|_{A}\|y\|_{A})^2(\|x\|_{A}^2+(\|x\|_{A}+\|X\|_{A}\|y\|_{A})^2)\\
	&=& \sup_{\theta \in [0,\frac{\pi}{2}]}(\cos \theta+\|X\|_{A}\sin \theta)^2(\cos^2 \theta+(\cos \theta+\|X\|_{A}\sin \theta)^2).\\
	\end{eqnarray*}
First we consider the case $\|X\|_{A}=0.$ Then
$$\sup_{\theta \in [0,\frac{\pi}{2}]}(\cos \theta+\|X\|_{A}\sin \theta)^2(\cos^2 \theta+(\cos \theta+\|X\|_{A}\sin \theta)^2)=2.$$\\
Therefore, $dw_{\mathbb{A}}(\mathbb{T}) \leq \sqrt{2}.$ Now let $z=\left(\begin{array}{cc}
	x \\
	0 
	\end{array}\right)$ be such that $\|z\|_{\mathbb{A}}=1$, i.e., $\|x\|_{A}=1$. Then $ \langle \mathbb{T}z,z \rangle _{\mathbb{A}}= \|x\|_{A}^2 $ and $ \langle \mathbb{T}z,\mathbb{T}z \rangle_{\mathbb{A}} = \|x\|_{A}^2.$ Hence, $(|\langle \mathbb{T}z,z \rangle_{\mathbb{A}}|^2+|\langle \mathbb{T}z,\mathbb{T}z \rangle_{\mathbb{A}}|^2)^\frac{1}{2}=\sqrt{2}.$ Therefore, $dw_{\mathbb{A}}(\mathbb{T}) = \sqrt{2}.$
		
\noindent Next we consider the case $\|X\|_{A}\neq 0.$ Then
\begin{eqnarray*}
	&& \sup_{\theta \in [0,\frac{\pi}{2}]}(\cos \theta+\|X\|_{A}\sin \theta)^2(\cos^2 \theta+(\cos \theta+\|X\|_{A}\sin \theta)^2)\\
	&&=	(\cos \theta_0 + \|X\|_{A} \sin \theta_0)^2(\cos^2 \theta_0 +(\cos \theta_0 + \|X\|_{A}\sin \theta_0)^2),
	\end{eqnarray*}
where $b=\|X\|_{A}$, $p=-\frac{2b^2-5}{2b},$ $q =- \frac{2b^2-2}{b^2},$ $r= -\frac{3}{2b},$ $s= \frac{1}{2^43^3b^6}(8b^8+20b^6+45b^4+61b^2+28),$ $\alpha=\frac{1}{27}(2p^3-9pq+27r),$ $\beta =(-\frac{\alpha}{2}+\sqrt{s})^\frac{1}{3},$ $\gamma=(-\frac{\alpha}{2}-\sqrt{s})^\frac{1}{3}$ and $\theta_0 = \tan^{-1}(\beta+\gamma-\frac{p}{3}).$
Therefore, 
\[dw_{\mathbb{A}}(\mathbb{T}) \leq (\cos \theta_0 + \|X\|_{A} \sin \theta_0)(\cos^2 \theta_0 +(\cos \theta_0 + \|X\|_{A}\sin \theta_0)^2)^\frac{1}{2} .\]
 We now show that there exists a sequence $\{z_n\}$ in $\mathcal{H} \oplus \mathcal{H}$ with $\|z_n\|_{\mathbb{A}}=1$ such that $ \lim_{n \to \infty}(|\langle \mathbb{T}z_n,z_n \rangle_{\mathbb{A}}|^2+|\langle \mathbb{T}z_n,\mathbb{T}z_n \rangle_{\mathbb{A}}|^2)^\frac{1}{2}=(\cos \theta_0 + \|X\|_{A} \sin \theta_0)(\cos^2 \theta_0 +(\cos \theta_0 + \|X\|_{A}\sin \theta_0)^2)^\frac{1}{2}.$
 Since $X\in \mathcal{B}_{A^{1/2}}(\mathcal{H})$, there exists a sequence $\{y_n\}$ in $\mathcal{H}$ with $\|y_n\|_{A}=1$ such that $\lim_{n \to \infty} \|Xy_n\|_{A} = \|X\|_{A}.$
 Let $z^k_{n}=\frac{1}{\sqrt{\|Xy_n\|_{A}^2+k^2}}\left(\begin{array}{cc}
	Xy_n \\
	ky_n 
	\end{array}\right)$, where $k \geq 0 $. Then
	$|\langle \mathbb{T}z^k_{n},z^k_{n} \rangle_{\mathbb{A}}|^2+|\langle \mathbb{T}z^k_{n},\mathbb{T}z^k_{n} \rangle_{\mathbb{A}}|^2  = \frac{(1+k)^2\|Xy_n\|_{A}^4}{(\|Xy_n\|_{A}^2+k^2)^2}\left (1+(1+k)^2\right)$

	$= \left (\frac{\|Xy_n\|_{A}}{\sqrt{\|Xy_n\|_{A}^2+k^2}}+\frac{k\|Xy_n\|_{A}}{\sqrt{\|Xy_n\|_{A}^2+k^2}} \right )^2
	 \left (\frac{\|Xy_n\|_{A}^2}{\|Xy_n\|_{A}^2+k^2}+\left (\frac{\|Xy_n\|_{A}}{\sqrt{\|Xy_n\|_{A}^2+k^2}}+\frac{k\|Xy_n\|_{A}}{\sqrt{\|Xy_n\|_{A}^2+k^2}} \right )^2 \right ).$
We can choose $k_0 \geq 0$ such that  $\frac{\|X\|_{A}}{\sqrt{\|X\|_{A}^2+k_0^2}} = \cos \theta_0$ and  $ \frac{k_0}{\sqrt{\|X\|_{A}^2+k_0^2}}= \sin \theta_0$.
Therefore, if we choose $z_n=\frac{1}{\sqrt{\|Xy_n\|_{A}^2+k_0^2}}\left(\begin{array}{cc}
	Xy_n \\
	k_0y_n 
	\end{array}\right)$, then $\lim_{n \to \infty}(|\langle \mathbb{T}z_n,z_n \rangle_{\mathbb{A}}|^2+|\langle \mathbb{T}z_n,\mathbb{T}z_n \rangle_{\mathbb{A}}|^2)^\frac{1}{2}$ $= \Big (\cos \theta_0 + \|X\|_{A} \sin \theta_0 \Big)\Big (\cos^2 \theta_0 +(\cos \theta_0 + \|X\|_{A}\sin \theta_0)^2\Big)^\frac{1}{2}.$ 
This completes the proof.
\end{proof}

\begin{theorem}\label{32}
Let $X \in \mathcal{B}_{A^{1/2}}(\mathcal{H})$ and  $\mathbb{S} =\left(\begin{array}{cc}
	O & X\\
	O & O
	\end{array}\right)$. Then 
	\[ dw_{\mathbb{A}}(\mathbb{S})=\begin{cases}
		0, &\|X\|_{A}=0 \\
		\frac{\|X\|_{A}}{2\sqrt{1-\|X\|_{A}^2}}, & \|X\|_{A} < \frac{1}{\sqrt{2}} \\
		\|X\|_{A}^2, & \|X\|_{A} \geq \frac{1}{\sqrt{2}}.
		\end{cases} \]	
\end{theorem}
\begin{proof}
 Let  $z=\left(\begin{array}{cc}
	x \\
	y 
	\end{array}\right) \in \mathcal{H}\oplus \mathcal{H} $ be such that $\|z\|_{\mathbb{A}}=1$, i.e, $\|x\|_{A}^2+\|y\|_{A}^2=1$.
Then $ \langle \mathbb{S}z,z \rangle_{\mathbb{A}} = \langle Xy,x \rangle_{A} ~~\mbox{and} ~~ \langle \mathbb{S}z,\mathbb{S}z \rangle_{\mathbb{A}} = \langle Xy,Xy \rangle_{A}.$ Now we have
\begin{eqnarray*}
	|\langle \mathbb{S}z,z \rangle_{\mathbb{A}}|^2+|\langle \mathbb{S}z,\mathbb{S}z \rangle_{\mathbb{A}}|^2 &\leq& \|Xy\|_{A}^2\|x\|_{A}^2 +\|Xy\|_{A}^4 \\
	&\leq & \sup_{\|x\|_{A}^2+\|y\|_{A}^2=1} \left (\|X\|_{A}^2\|y\|_{A}^2\|x\|_{A}^2 + \|X\|_{A}^4\|y\|_{A}^4 \right)\\
	&=& \sup_{\theta \in [0,\frac{\pi}{2}]} \|X\|_{A}^2 \sin^2 \theta \left ( \cos^2 \theta +\|X\|_{A}^2 \sin^2 \theta \right ).
	\end{eqnarray*}
	First we consider the case $\|X\|_{A}=0.$ Then it is easy to see that $dw_{\mathbb{A}}(\mathbb{S})=0.$

	\noindent Next we consider the case  $ 0 < \|X\|_{A} < \frac{1}{\sqrt{2}}.$
	Then  $$\sup_{\theta \in [0,\frac{\pi}{2}]} \|X\|_{A}^2 \sin^2 \theta \left ( \cos^2 \theta +\|X\|_{A}^2 \sin^2 \theta \right )=\frac{\|X\|_{A}^2}{4(1-\|X\|_{A}^2)}.$$
Therefore, $dw_{\mathbb{A}}(\mathbb{S}) \leq \frac{\|X\|_{A}}{2\sqrt{(1-\|X\|_{A}^2)}}.$ We now show that there exists a sequence $\{z_n\}$ in $\mathcal{H} \oplus \mathcal{H}$ with $\|z_n\|_{\mathbb{A}}=1$ such that 
$$ \lim_{n \to \infty}\{|\langle \mathbb{S}z_n,z_n \rangle_{\mathbb{A}}|^2+|\langle \mathbb{S}z_n,\mathbb{S}z_n \rangle_{\mathbb{A}}|^2\}^\frac{1}{2}=
	\frac{\|X\|_{A}}{2\sqrt{(1-\|X\|_{A}^2)}}  .$$

\noindent Since $X\in \mathcal{B}_{A^{1/2}}(\mathcal{H})$, there exists a sequence $\{y_n \}$ in $\mathcal{H}$ with $\|y_n\|_{A}=1$ such that $\lim_{n \to \infty} \|Xy_n\|_{A} = \|X\|_{A}.$
Let $z_{n}=\frac{1}{\sqrt{\|Xy_n\|_{A}^2+k^2}}\left(\begin{array}{cc}
	Xy_n \\
	ky_n 
	\end{array}\right)$, where $k =  \frac{\|X\|_{A}}{\sqrt{1-2\|X\|_{A}^2}}$. Then
	\begin{eqnarray*}
	\lim_{n\rightarrow \infty} \{|\langle \mathbb{S}z_{n},z_{n} \rangle_{\mathbb{A}}|^2+|\langle \mathbb{S}z_{n},\mathbb{S}z_{n} \rangle_{\mathbb{A}}|^2\} ^{\frac{1}{2}}&=& \frac{\|X\|_{A}}{2\sqrt{1-\|X\|_{A}^2}}.
	\end{eqnarray*}
Therefore, $dw_{\mathbb{A}}(\mathbb{S})=\frac{\|X\|_{A}}{2\sqrt{(1-\|X\|_{A}^2)}}.$

\noindent Now we consider the case  $\|X\|_{A} \geq \frac{1}{\sqrt{2}}$. Then $$\sup_{\theta \in [0,\frac{\pi}{2}]} \|X\|_{A}^2 \sin^2 \theta \left ( \cos^2 \theta +\|X\|_{A}^2 \sin^2 \theta \right )=\|X\|_{A}^4.$$
Therefore, $dw_{\mathbb{A}}(\mathbb{S}) \leq \|X\|_{A}^2. $ We now show that there exists a sequence $\{z_n\}$ in $\mathcal{H} \oplus \mathcal{H}$ with $\|z_n\|_{\mathbb{A}}=1$ such that 
$$ \lim_{n \to \infty}(|\langle \mathbb{S}z_n,z_n \rangle_{\mathbb{A}}|^2+|\langle \mathbb{S}z_n,\mathbb{S}z_n \rangle_{\mathbb{A}}|^2)^\frac{1}{2}=
	\|X\|_{A}^2 .$$
\noindent Since $X\in \mathcal{B}_{A^{1/2}}(\mathcal{H})$, there exists a sequence $\{y_n \}$ in $\mathcal{H}$ with $\|y_n\|_{A}=1$ such that $\lim_{n \to \infty} \|Xy_n\|_{A} = \|X\|_{A}.$	
 If we consider $z_n=\left(\begin{array}{cc}
	0 \\
	y_n 
	\end{array}\right)$, then $ \langle \mathbb{S}z_n,z_n \rangle_{\mathbb{A}} = 0 $ and $ \langle \mathbb{S}z_n,\mathbb{S}z_n \rangle_{\mathbb{A}} = \|Xy_n\|_{A}^2.$ Therefore, $\lim_{n \to \infty}(|\langle \mathbb{S}z_n,z_n \rangle_{\mathbb{A}}|^2+|\langle \mathbb{S}z_n,\mathbb{S}z_n \rangle_{\mathbb{A}}|^2)^\frac{1}{2}=\|X\|_{A}^2.$ This completes the proof.
		\end{proof}

\begin{remark}
In particular, if we consider $A=I$ in Theorem \ref{31} and Theorem \ref{32} then we get \cite[Th. 3.14]{BBBP} and \cite[Th. 3.16]{BBBP}, respectively.
\end{remark}

\bibliographystyle{amsplain}

\begin{thebibliography}{99}

%\bibitem{acg3} {M.L. Arias, G. Corach and M.C. Gonzalez,} {Lifting properties in operator ranges,} Acta Sci. Math. (Szeged) 75:3-4(2009) 635-653.

\bibitem{acg1}{M.L. Arias, G. Corach and M.C. Gonzalez,} {Partial isometries in semi-Hilbertian spaces,} Linear Algebra Appl. 428 (7) (2008) 1460-1475.

\bibitem{acg2} {M.L. Arias, G. Corach and M.C. Gonzalez,} {Metric properties of projections in semi-Hilbertian spaces,} Integral Equations Operator Theory 62 (2008) 11-28.

\bibitem{BFS} H. Baklouti, K.Feki, O.A.M. Sid Ahmed, Joint normality of operators in semi-Hilbertian spaces, Linear Multilinear Algebra 68(4) (2020) 845-866.

\bibitem{BBBP} P. Bhunia, A. Bhanja, S. Bag and K. Paul, Bounds for the Davis-Wielandt radius of bounded linear operators, \url{http://arxiv.org/abs/2006.04389}. 

\bibitem{BNP} P. Bhunia, R.K. Nayak and K. Paul, Refinements of A-numerical radius inequalities and their applications, Adv. Oper. Theory (2020). \url{https://doi.org/10.1007/s43036-020-00056-8}.

\bibitem{BFP} P. Bhunia, K. Feki and K. Paul, A-Numerical radius orthogonality and parallelism of semi-Hilbertian space operators and their applications, Bull. Iran. Math. Soc. (2020). \url{https://doi.org/10.1007/s41980-020-00392-8}.

\bibitem{D} C. Davis, The shell of a Hilbert-space operator, Acta Sci. Math., (Szeged) 29 (1968) 69-86.


\bibitem{doug} R.G. Douglas, On majorization, factorization and range inclusion of operators in Hilbert space, Proc. Amer. Math. Soc. 17 (1966) 413-416.

\bibitem{F} K. Feki, Spectral radius of semi-Hilbertian space operators and its applications, Ann. Funct. Anal. (2020). \url{https://doi.org/10.1007/s43034-020-00064-y}.

\bibitem{F2} K. Feki, Inequalities for the A-joint numerical radius of two operators and their applications, arXiv:2005.04758v1 [math.FA] 10 May (2020).

\bibitem{FA} K. Feki and O.A.M. Sid Ahmed, Davis-Wielandt shells of semi-Hilbertian space operators and its applications, Banach J. Math. Anal. (2020), \url{https://doi.org/10.1007/s43037-020-00063-0}.


\bibitem{LP} C.-K. Li and Y.-T. Poon, Spectrum, numerical range and Davis-Wielandt Shell of a normal operator, Glasgow Math. J. 51 (2009) 91-100.

\bibitem{LPS}  C.-K. Li, Y.-T. Poon and N.S. Sze, Davis-Wielandt shells of operators, Oper. Matrices,  2(3) (2008) 341-355, \url{https://dx.doi.org/10.7153/oam-02-20}.

\bibitem{LSZ} B. Lins, I.M. Spitkovsky, S. Zhong, The normalized numerical range and the Davis-Wielandt shell, Linear Algebra Appl. 546 (2018) 187-209.

\bibitem{Z} A. Zamani, A-numerical radius inequalities for semi-Hilbertian space operators, Linear Algebra Appl. 578 (2019) 159-183.

\bibitem{ZS} A. Zamani and K. Shebrawi, Some Upper Bounds for the Davis-Wielandt Radius of Hilbert Space Operators, Mediterr. J. Math., (2019), \url{https://doi.org/10.1007/s00009-019-1458-z}.

\bibitem{ZMCN}  A. Zamani, M.S. Moslehian, M.-T. Chien and H. Nakazato, Norm-parallelism and the Davis-Wielandt radius of Hilbert space operators, Linear
Multilinear Algebra, 67(11), (2019) 2147-2158.

\bibitem{W} H. Wielandt, On eigenvalues of sums of normal matrices, Pac. J. Math. 5 (1955) 633-638. 
 
\end{thebibliography}

\end{document}